\documentclass[a4paper,12pt]{article}
\usepackage{xargs,hyperref}
\usepackage[pdftex,dvipsnames]{xcolor}
\usepackage{comment,amsmath,amssymb,amsthm,changepage,pifont,graphicx,faktor,tikz}
\usepackage[english]{babel}
\usetikzlibrary{patterns}
\usepackage[colorinlistoftodos,prependcaption,textsize=tiny]{todonotes}
\newcommandx{\unsure}[2][1=]{\todo[linecolor=red,backgroundcolor=red!25,bordercolor=red,#1]{#2}}
\newcommandx{\info}[2][1=]{\todo[linecolor=OliveGreen,backgroundcolor=OliveGreen!25,bordercolor=OliveGreen,#1]{#2}}
\usepackage[font=small]{caption}

\newtheorem{theorem}{Theorem}[section]

\newtheorem{lemma}[theorem]{Lemma}
\newtheorem{corollary}[theorem]{Corollary}

\newtheorem{definition}[theorem]{Definition}

\theoremstyle{remark}
\newtheorem{remark}[theorem]{Remark}

\DeclareMathOperator{\lw}{lw}
\DeclareMathOperator{\ls}{ls}
\DeclareMathOperator{\conv}{conv}

\DeclareMathOperator{\vol}{vol}
\DeclareMathOperator{\GL}{GL}

\newcommand{\RR}{\mathbb{R}}
\newcommand{\ZZ}{\mathbb{Z}}

\begin{document}

\title{Minimal polygons with fixed lattice width}
\author{F.\ Cools and A.\ Lemmens}
\date{}
\maketitle

\begin{abstract}
\noindent We classify the unimodular equivalence classes of inclusion-minimal polygons with a certain fixed lattice width. 
As a corollary, we find a sharp upper bound on the number of lattice points of these minimal polygons. \\

\end{abstract}

\section{Introduction and definitions} \label{sec_1}

Let $\Delta\subset \RR^2$ be a non-empty \emph{lattice polygon}, i.e. the convex hull of a finite number of lattice points in $\ZZ^2$, and consider a \emph{lattice direction} $v\in \ZZ^2$, i.e. a non-zero primitive vector. The \emph{lattice width of $\Delta$ in the direction $v$} is $$\lw_v(\Delta)=\max_{P\in\Delta}\langle P,v\rangle-\min_{P\in\Delta}\langle P,v\rangle.$$ The \emph{lattice width} of $\Delta$ is defined as $\lw(\Delta)=\min_v \lw_v(\Delta)$. Throughout this paper we will assume that $\Delta$ is two-dimensional, hence $\lw(\Delta)>0$. A lattice direction $v$ that satisfies $\lw_v(\Delta)=\lw(\Delta)$ is called a \emph{lattice width direction} of $\Delta$.  

Two lattice polygons $\Delta$ and $\Delta'$ are called \emph{(unimodularly) equivalent} if and only if there exists a \emph{unimodular transformation} $\varphi$, i.e. a map of the form 
$$\varphi:\RR^2\to\RR^2:x\mapsto Ax+b, \quad \text{where} \quad A\in \GL_2(\ZZ),\ b\in \ZZ^2,$$ 
such that $\varphi(\Delta)=\Delta'$. Equivalent lattice polygons have the same lattice width. 

The lattice width of a polygon can be seen as a specific instance of the more general notion of \emph{lattice size}, which was introduced in \cite{CaCo2}.
\begin{definition}
Let $X\subset \RR^2$ be a subset with positive Jordan measure. Then the lattice size $\ls_X(\Delta)$ of a non-empty lattice polygon $\Delta$ is the smallest $d\in\ZZ_{\geq 0}$ for which there exists a unimodular transformation $\varphi$ such that $\varphi(\Delta)\subset dX$. 
\end{definition}
Note that $\lw(\Delta)=\ls_X(\Delta)$, where $X=\RR\times[0,1]$. 

This paper is concerned with polygons $\Delta$ that are \emph{minimal} in the following sense: $\lw(\Delta')<\lw(\Delta)$ for each lattice polygon 
$\Delta'\subsetneq \Delta$. Equivalently, a two-dimensional polygon $\Delta$ is minimal if and only if for each vertex $P$ of $\Delta$, we have that $\lw(\Delta_P)<\lw(\Delta)$, where $\Delta_P:=\conv((\Delta\cap\ZZ^2)\setminus\{P\})$.

Our main result is a complete classification of minimal polygons up to unimodular equivalence, see Theorem \ref{thm_clas}. As a corollary, we provide a sharp upper bound on the number of lattice points of these minimal polygons. First, we show in Lemma \ref{lemma2} that each minimal polygon $\Delta$ satisfies $\ls_\square(\Delta)=\lw(\Delta)$, where $\square=\conv\{(0,0),(1,0),(1,1),(0,1)\}$. The latter can also be proven using results on lattice width directions of interior lattice polygons (see \cite[Lemma 5.3]{CCDL}), but we choose to keep the paper self-contained and have provided a different proof. 
Moreover, we use the technical Lemma \ref{newlemma} in the proofs of both Lemma \ref{lemma2} and Theorem \ref{thm_clas}. 

In the joint paper \cite{CCDL} with Castryck and Demeyer, we study the Betti table of the toric surface $\text{Tor}(\Delta)\subset \mathbb{P}^{\sharp(\Delta\cap\ZZ^2)-1}$ for lattice polygons $\Delta$. In particular, we present a lower bound for the length of the linear strand of this Betti table in terms of $\lw(\Delta)$, which we conjecture to be sharp. For showing this conjecture for polygons of a fixed lattice width, it essentially suffices to prove it for the minimal polygons (see \cite[Corollary 5.2]{CCDL}). Hence, Theorem \ref{thm_clas} allows us to check the conjecture using a computer algebra system.

\begin{remark}
Of course, the question of classifying minimal polytopes can also be asked in higher dimensions. For instance, it can be shown that each three-dimensional minimal polytope $\Delta\subset \RR^3$ with $\lw(\Delta)=1$ is equivalent to a tetrahedron of the form 
$$\conv\{(0,0,0),(1,0,0),(0,1,0),(1,y,z)\}$$ with $1\leq y\leq z$ and $\gcd(y,z)=1$. These include the Reeve tetrahedrons (where $y=1$). For comparison, there is only one minimal polygon with lattice width one up to equivalence, namely the standard simplex $\conv\{(0,0),(1,0),(0,1)\}$. 

In all dimensions $k\geq 2$, among the minimal polytopes we will find back the so-called \emph{empty lattice simplices} $\Delta\subset \RR^k$, i.e. convex hulls of $k+1$ lattice points without interior lattice points. If $k\geq 4$, not all empty lattice simplices have lattice width $1$. For more information, see 
\cite{BBBK,HZ,Seb}.
\end{remark}

\section{The classification of minimal polygons}

Throughout this section, we will use the notations from Section \ref{sec_1}. The following result appears already in \cite[Remark following Lemma 5.2]{CaCo1}, but can be proven in a shorter way.  

\begin{lemma} \label{lemma1}
Let $\Delta\subset\RR^2$ be a lattice polygon with $\lw(\Delta)=d$. If $\Delta$ has two linearly independent lattice width directions $v,w\in\ZZ^2$, then $ls_\square(\Delta)=d$.
\end{lemma}

\begin{proof}
If $v$ and $w$ do not form a $\ZZ$-basis of $\ZZ^2$, we take a primitive vector $u\in\conv\{(0,0),v,w\}$ such that $v$ and $u$ form a $\ZZ$-basis. Let $Q,Q'$ be lattice points of $\Delta$ such that $\langle Q',u\rangle-\langle Q,u\rangle=\lw_u(\Delta)$. Write $u=\lambda v+\mu w$ with $0<\lambda,\mu$ and $\lambda+\mu\leq 1$. Now $$d\leq \lw_u(\Delta)=\langle Q',(\lambda v+\mu w)\rangle-\langle Q,(\lambda v+\mu w)\rangle\leq\lambda \lw_v(\Delta)+\mu \lw_w(\Delta)\leq d,$$
so $\lw_u(\Delta)=d$. After applying a unimodular transformation, we may assume that $u=(0,1)$ and $v=(1,0)$, and that $\Delta$ fits into $d\square$, hence $ls_\square(\Delta)=d$.
\end{proof}

\begin{lemma} \label{newlemma}
Let $\Delta$ be a lattice polygon with $\lw(\Delta)=d>0$. Let $P$ be a vertex of $\Delta$ and $v\in\ZZ^2$ be a primitive vector. If $\lw_v(\Delta_P)<d$ and $\lw_v(\Delta_P)<\lw_v(\Delta)-1$, then $\Delta$ is equivalent to $\Upsilon_{d-1}:=\conv\{(0,0),(1,d),(d,1)\}$.
\end{lemma}

\begin{proof}
Since $\lw_v(\Delta_P)<\lw_v(\Delta)-1$, we have that either 
$$\min_{Q\in\Delta_P}\langle v,Q\rangle > \langle v,P\rangle+1 \quad \text{or} \quad
\max_{Q\in\Delta_P}\langle v,Q\rangle < \langle v,P\rangle-1.$$
By replacing $v$ by $-v$, we may assume that we are in the first case. Moreover, we may choose $v$ such that the difference $\min_{Q\in\Delta_P}\langle v,Q\rangle-\langle v,P\rangle$ is minimal but greater than 1, and such that $\lw_v(\Delta_P)<d$.

We apply a unimodular transformation so that $P=(0,0)$ and $v=(0,1)$. Let $y_m$ (resp. $y_M$) be the smallest (resp. greatest) $y$-coordinate occurring in $\Delta_P$. Note that $y_m=\min_{Q\in\Delta_P}\langle v,Q\rangle$ and $y_M=\max_{Q\in\Delta_P}\langle v,Q\rangle$, hence $y_m>1$ and $y_M-y_m<d$. 

Define the cone $C_k:=\{\lambda (k,1)+\mu (k+1,1)|\lambda,\mu\geq 0\}$. Since 
$$\Delta \subset (\RR\times\RR_{>0})\cup\{P\} = \cup_{k\in\ZZ}\, C_k$$ 
and $y_m>1$, the polygon $\Delta$ is contained in a cone $C_k$ for some $k\in\ZZ$. Using the unimodular transformation $(x,y)\mapsto(x-ky,y)$, we may assume that $k=0$, i.e. $\Delta\subseteq C_0=\{\lambda (0,1)+\mu (1,1)|\lambda,\mu\geq 0\}$. In fact, we then have
 that $$\Delta\subseteq\conv\{(0,0),(1,y_M),(y_M-1,y_M)\}.$$

If $y_m=2$, we have $y_M=(y_M-y_m)+2\leq d+1$. The strict inequality $y_M<d+1$ is impossible as the horizontal width $\lw_{(1,0)}(\Delta)$ would be less than $d$. So we have that $y_M=d+1$ and $\Delta\subseteq \Delta'=\conv\{(0,0),(1,d+1),(d,d+1)\}$. Since $\lw((\Delta')_Q)<d$ for $Q\in\{(1,d+1),(d,d+1)\}$, we must have $\Delta=\Delta'$. This is equivalent to $\Upsilon_{d-1}$ via $(x,y)\mapsto(x,y-x)$.

From now on, assume that $y_m>2$. Then $(1,2)\notin\Delta$ which means that either
$$\Delta\subseteq\{\lambda(0,1)+\mu(1,2)|\lambda,\mu\geq 0\}\quad \text{or} \quad \underline{\Delta\subseteq\{\lambda(1,2)+\mu(1,1)|\lambda,\mu\geq 0\}}.$$
We can reduce to the latter case using the transformation $(x,y)\mapsto(y-x,y)$. In fact, we can keep subdividing this cone until we find a cone $C$ containing $\Delta$ that does not contain any lattice point with $y$-coordinate in $\{1,\ldots,y_m-1\}$. Let $\ell\in\ZZ$ be such that $C$ passes in between $(\ell-1,y_m-1)$ and $(\ell,y_m-1)$. 
Then $$\Delta_P\subseteq\conv\{(\ell,y_m-1),(\ell,y_M),(\ell+y_M-y_m+1,y_M)\}.$$ If $x_m$ (resp. $x_M$) is the smallest (resp. greatest) $x$-coordinate occurring in a lattice point of $\Delta_P$, then $2\leq \ell\leq x_m<y_m$ and $x_M\leq \ell+y_M-y_m$, so $x_M-x_m\leq y_M-y_m<d$. But this means that $\lw_{(1,0)}(\Delta_P)<d$ and 
$$1<\min_{Q\in\Delta_P}\langle(1,0),Q\rangle < y_m = \min_{Q\in\Delta_P}\langle v,Q\rangle$$
contradicting the minimality of $v$.
\end{proof}

\begin{lemma} \label{lemma2}
If $\Delta\subset \RR^2$ is a non-empty minimal lattice polygon with $\lw(\Delta)=d>0$, then $\ls_\square(\Delta)=d$. 
\end{lemma}

\begin{proof}
By Lemma \ref{lemma1}, we only have to show that there are two linearly independent lattice width directions. Suppose that $v$ is a lattice width direction and that $Q,Q'\in\Delta\cap\ZZ^2$ such that $\langle Q,v\rangle-\langle Q',v\rangle=d$. Now let $P$ be a vertex of $\Delta$ different from $Q,Q'$. By minimality of $\Delta$, we have that $\lw(\Delta_P)<d$. That means there exists a direction $w$ such that $\lw_w(\Delta_P)<d$. Because $Q$ and $Q'$ are still in $\Delta_P$, $w$ cannot be $v$ or $-v$, so $w$ must be linearly indepentent of $v$. If $\lw_w(\Delta)=d$, we are done. If $\lw_w(\Delta)>d$, then by Lemma \ref{newlemma}, $\Delta$ is equivalent to $\Upsilon_{d-1}\subseteq d\square$.
\end{proof}

\begin{theorem} \label{thm_clas}
Let $\Delta\subset \RR^2$ be a non-empty minimal lattice polygon with $\lw(\Delta)=d$. Then $\Delta$ is equivalent to a minimal polygon of one of the following forms:
\begin{enumerate}
\item[(T1)] $\conv\{(0,0),(d,y),(x,d)\}$ where $x,y\in\{0,\ldots,d\}$ satisfy $x+y\leq d$;
\item[(T2)] $\conv\{(x_1,0),(d,y_2),(x_2,d),(0,y_1)\}$ where $x_1,x_2,y_1,y_2\in\{1,\ldots,d-1\}$ satisfy $\max(x_2,y_2)\geq\min(x_1,y_1)$ and $\max(d-x_2,y_1)\geq\min(d-x_1,y_2)$;
\item[(T3)] $\conv\{(0,0),(\ell,0),(d,y+d-\ell),(x+\ell,d),(z,z+d-\ell)\}$ with $\ell\in\{2,\ldots,d-2\}$, $x\in\{1,\ldots,d-\ell-1\}$, $y,z\in\{1,\ldots,\ell-1\}$;
\item[(T4)] $\conv\{(0,0),(z'+\ell,z'),(d,y+d-\ell),(x+\ell,d),(z,z+d-\ell)\}$ with $\ell\in\{2,\ldots,d-2\}$, $y,z\in \{1,\ldots,\ell-1\}$, $x,z'\in\{1,\ldots,d-\ell-1\}$;
\item[(T5)] $\conv\{(x_1,0),(z_2+\ell,z_2),(d,d-\ell+y_2),(x_2+\ell,d),(z_1,z_1+d-\ell),(0,y_1)\}$ with $\ell\in\{2,\ldots,d-2\}$, $x_1,y_2,z_1\in \{1,\ldots,\ell-1\}$, $x_2,y_1,z_2\in\{1,\ldots,d-\ell-1\}$;
\end{enumerate}
\end{theorem}

\begin{remark} \label{remarkH}
See Figure \ref{figure types} for a picture of the five types.
The minimal polygons appearing in the types (T3), (T4) and (T5) are inscribed in the hexagon 
$$H_\ell:=\conv\{(0,0),(\ell,0),(d,d-\ell),(d,d),(\ell,d),(0,d-\ell)\}.$$ 
This is also the case for the triangles of type (T1) with $(x,y)\in\{(d,0),(0,d)\}$ (where we allow $\ell\in\{0,d\}$) 
and for the quadrangles of type $(T2)$ with $\max(d-x_2,y_1)=\min(d-x_1,y_2)$. 

\begin{figure}[h!]
\centering
\includegraphics[height=4cm]{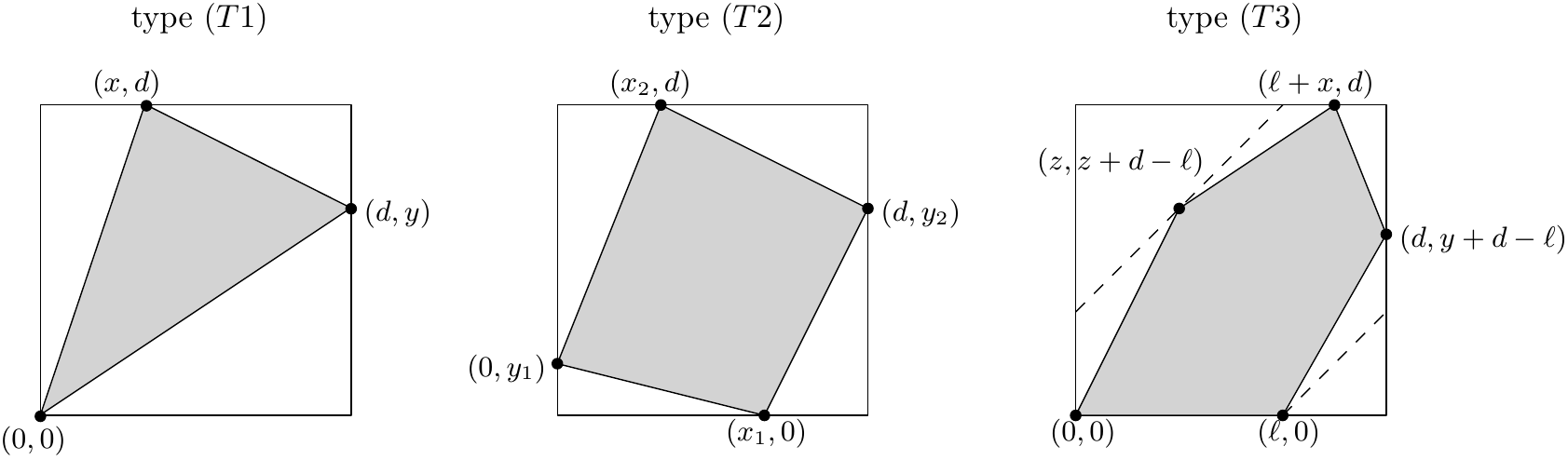}
\includegraphics[height=4cm]{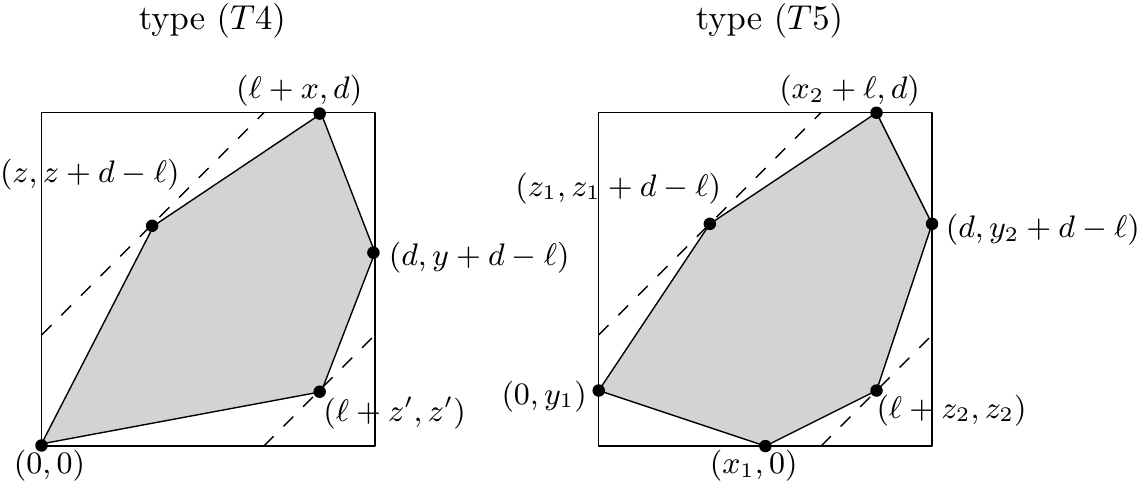}
\caption{The five types in the classification}
\label{figure types}
\end{figure}
\end{remark}

\begin{proof}
If $d=0$, then $\Delta$ consists of a single point and it is of shape $(T1)$. So assume $d\geq 1$. Because of Lemma \ref{lemma2}, we may assume that $\Delta\subset d\square=[0,d]\times [0,d]$. Moreover, we may assume that $\Delta\not\cong\Upsilon_{d-1}$ since $\Upsilon_{d-1}$ is of type $(T1)$. Let $P$ be any vertex of $\Delta$. By Lemma \ref{newlemma}, if $\lw_v(\Delta_P)<d$ for some primitive vector $v\in\ZZ^2$, then $\lw_v(\Delta_P)=d-1$ and $\lw_v(\Delta)=d$, hence $v$ is a lattice width direction. 

By minimality, we know that there always exists a lattice direction $v$ satisfying $\lw_v(\Delta_P)<d$. We claim that we can always take $v\in\{(0,1),(1,0),(1,1),(1,-1)\}$. Indeed, suppose that $v=(v_x,v_y)\in\ZZ^2$ satisfies $$\{v,-v\}\cap\{(0,1),(1,0),(1,1),(1,-1)\}=\emptyset \quad \text{and} \quad \lw_v(\Delta_P)<d.$$ After a unimodular transformation, we may assume that $0<v_x<v_y$, hence $(1,1)\in\conv\{(0,0),(1,0),v\}$. Using a similar trick as in Lemma \ref{lemma1}, we get that $\lw_{(1,1)}(\Delta_P)<d$, which proves the claim.   

Let $\mathcal{V}$ be set consisting of vectors $v\in \{(1,1),(1,-1)\}$ for which there exists a vertex $P$ of $\Delta$ with $\lw_v(\Delta_P)<d$. If $\mathcal{V}=\{(1,1),(1,-1)\}$, then $\Delta$ has $4$ different lattice width directions, namely $(1,0),(0,1),(1,1)$ and $(1,-1)$. By \cite[Lemma 5.2(v)]{CaCo1} or \cite{DMN}, this means that $\Delta\cong \conv\{(d/2,0),(0,d/2),(d/2,d),(d,d/2)\}$ for some even $d$, hence it is of type $(T2)$. 
If $\mathcal{V}=\emptyset$, we claim that $\Delta$ is of type $(T1)$ or $(T2)$. 
Indeed, for every vertex $P$ of $\Delta$, we have that either $\lw_{(1,0)}(\Delta_P)$ or $\lw_{(0,1)}(\Delta_P)$ is smaller than $d$. In particular, this means that there has to be a side of $d\square$ with $P$ as its only point in $\Delta$. One then easily checks the claim:
if $\Delta$ is a triangle, then it will be of type $(T1)$; if it is a quadrangle, then it is of type $(T2)$. 

From now on, suppose that $\mathcal{V}$ is not equal to $\emptyset$ or $\{(1,1),(1,-1)\}$, hence $\mathcal{V}=\{(1,1)\}$ or $\mathcal{V}=\{(1,-1)\}$.  We can suppose that $\mathcal{V}=\{(1,-1)\}$ by using the transformation $(x,y)\mapsto(x,-y)$ if necessary. 
Hence, for each vertex $P$ of $\Delta$, there is a vector $v\in\{(1,0),(0,1),(1,-1)\}$ with $\lw_v(\Delta_P)<d$. 
Since $\lw_{(1,-1)}(\Delta)=d$, there exists an integer $\ell\in\{0,\ldots,d\}$ such that $\langle Q,(1,-1)\rangle\in[\ell-d,\ell]$ for all $Q\in\Delta$. If $\ell\in\{0,d\}$, then $\Delta$ is a triangle whose vertices are vertices of $d\square$, so it is of the form $(T1)$. Now assume that $\ell\in\{1,\ldots,d-1\}$, hence $\Delta$ is contained in the hexagon $H_\ell$ from Remark \ref{remarkH}. 
Each side of $H_\ell$ contains at least one lattice point of $\Delta$, and if it contains more than one point, it is also an edge of $\Delta$. Otherwise, there would be a vertex $P$ lying on exactly one side of $H_\ell$, while not being the only point of $\Delta$ on that side of $H_\ell$. But then there is no $v\in\{(0,1),(1,0),(1,-1)\}$ with $\lw_v(\Delta_P)<d$ (as every side of $H_\ell$ contains a point of $\Delta_P$), a contradiction.

Denote by $\mathcal{S}$ the set of sides that $\Delta$ and $H_\ell$ have in common. 
Then $\mathcal{S}$ can not contain two adjacent sides $S_1,S_2$: otherwise for the vertex $P=S_1\cap S_2$, each side of $H_\ell$ would have a non-empty intersection with $\Delta_P$, contradicting the fact that there is a $v\in\{(0,1),(1,0),(1,-1)\}$ with $\lw_w(\Delta_P)<d$.

Assume that $\sharp \mathcal{S}\geq 2$ and take $S_1=[Q_1,Q_2]\in\mathcal{S}$. Its adjacent sides of $H_\ell$ contain no points of $\Delta$ except from $Q_1$ and $Q_2$. This implies that $\mathcal{S}=\{S_1,S_2\}$ where $S_1,S_2$ are opposite edges of $H_\ell$, and that $\Delta$ is the convex hull of these two edges. Hence $\Delta$ is equivalent to the quadrangle $\conv\{(\ell,0),(d,d-\ell),(\ell,d),(0,d-\ell)\}\subset H_\ell$, which is of type $(T2)$.  

If $\mathcal{S}$ consists of a single side $S$, we may assume that $S=[Q_1,Q_2]$ is the bottom edge of $H_\ell$. Let $P_1$ (resp. $P_2$) be the vertex of $\Delta$ on the upper left diagonal side (resp. the right vertical edge) of $H_\ell$. If $P_1$ is also on the top edge of $H_\ell$ (i.e. $P_1=(\ell,d)$), then $\Delta$ has only four vertices, namely $Q_1,Q_2,P_1,P_2$. Applying the transformation $(x,y)\mapsto (x,-x+y+\ell)$, we end up with a quadrangle of type $(T2)$. By a similar reasoning, if $P_2$ is on the top edge of $H_\ell$ (i.e. $P_2=(d,d)$), we end up with type $(T2)$. If neither $P_1$ nor $P_2$ are on the top edge of $H_\ell$, then there is a fifth vertex $P_3$ on that top edge, and we are in case $(T3)$.

The only remaining case is when $\mathcal{S}=\emptyset$, hence each edge of $H_\ell$ contains only one point of $\Delta$. If $H_\ell$ and $\Delta$ have no common vertex, then $\Delta$ is of type $(T5)$. If they share one vertex, we can reduce to type $(T4)$ using a transformation if necessary. Note that two common vertices of $H_\ell$ and $\Delta$ can never be connected by an edge of $H_\ell$ as that edge would be in $\mathcal{S}$, so there are at most three common vertices. If there are three shared vertices, then $\Delta$ is a triangle of type $(T1)$, again using a transformation if necessary. So assume $H_\ell$ and $\Delta$ share two vertices. Together these two points occupy four edges of $H_\ell$ and each of the other two edges of $H_\ell$ (call them $A$ and $B$) contains exactly one vertex of $\Delta$. Take two pairs of opposite sides of $H_\ell$ (so four sides in total) that together contain $A$ and $B$, then they contain all vertices of $\Delta$: since any common vertex of $H_\ell$ and $\Delta$ lies on two sides of $H_\ell$, they cannot lie both on the sides we didn't choose, as they are parallel. We can find a unimodular transformation mapping these sides into the four sides of $d\square$, hence $\Delta$ is of type $(T2)$. 
\end{proof}

\begin{remark}
From the classification in Theorem \ref{thm_clas}, one can easily deduce the following result from \cite{FTM}: $\vol(\Delta)\geq \frac{3}{8}\lw(\Delta)^2$ for each lattice polygon $\Delta\subset \RR^2$, and equality holds for minimal polygons of type $(T1)$ with $d$ even and $x=y=\frac{d}{2}$. For odd $d$, this inequality can be sharpened to $\vol(\Delta)\geq \frac{3}{8}\lw(\Delta)^2+\frac{1}{8}$, and equality holds for minimal polygons of type $(T1)$ with $x=\frac{d-1}{2}$ and $y=\frac{d+1}{2}$.
\end{remark}

\begin{corollary}
If $\Delta\subset\mathbb{R}^2$ is a non-empty minimal lattice polygon with $\lw(\Delta)=d>1$, then $$\sharp(\Delta\cap\ZZ^2)\leq \max((d-1)^2+4,(d+1)(d+2)/2).$$ Moreover, this bound is sharp.
\end{corollary}

\begin{proof}
Note that there exist minimal polygons attaining the bound (see Figure \ref{upperbound}): the simplex $\conv\{(0,0),(d,0),(0,d)\}$ is of type $(T1)$ and has $(d+1)(d+2)/2$ lattice points, and the quadrangle $\conv\{(1,0),(d,1),(d-1,d),(0,d-1)\}$ is of type $(T2)$ and has $(d-1)^2+4$ lattice points.

\begin{figure}[h!]
\centering
\includegraphics[height=3cm]{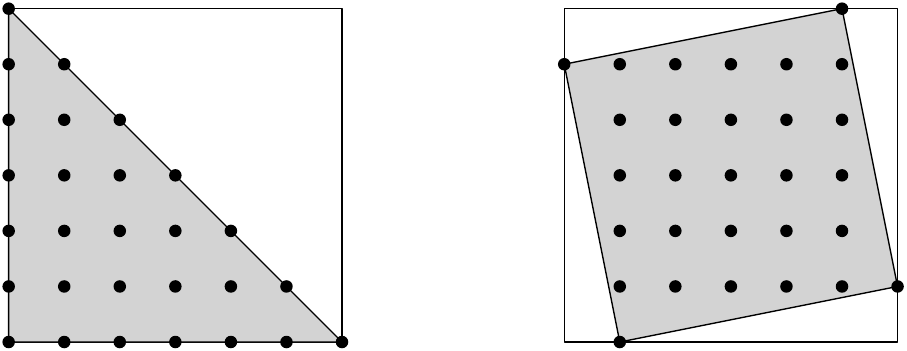}
\caption{Minimal polygons attaining the upper bound}
\label{upperbound}
\end{figure}
 
Now let's show that we indeed have an upper bound. If $\Delta$ is minimal of type $(T2)$, $(T4)$ or $(T5)$, then $\sharp(\Delta\cap\ZZ^2)\leq(d-1)^2+4$, since there are at most $4$ lattice points of $\Delta$ on the boundary of $d\square$ and all the others are in 
$$(d\square)^{\circ}\cap\ZZ^2 = \{1,\ldots,d-1\}\times \{1,\ldots,d-1\}.$$ 
This also holds for triangles of type $(T1)$ with $x$ and $y$ non-zero. If $\Delta$ is of type $(T3)$, we obtain the same upper bound $(d-1)^2+4$ after applying a unimodular transformation that maps the bottom edge of $\Delta$ to the left upper diagonal edge of $H_\ell$.
We are left with triangles of type $(T1)$ with either $x$ or $y$ zero. Assume that $y=0$ (the case $x=0$ is similar). Then $\Delta$ has the edge $[(0,0),(d,0)]$ in common with $d\square$ and its other vertex is $(x,d)$. For each $k\in\{0,\ldots,d\}$, the intersection of $\Delta$ with the horizontal line on height $k$ is a line segment of length $d-k$, hence it contains at most $d-k+1$ lattice points. 
So in total, $\Delta$ has at most $$\sum_{k=0}^d (d-k+1)=(d+1)(d+2)/2$$ lattice points. 
\end{proof}


\begin{thebibliography}{99}
\bibitem{BBBK} M. Barile, D. Bernardi, A. Borisov and J.-M. Kantor, \emph{On empty lattice simplices in dimension 4}, Proc. Amer. Math. Soc. 139 (2011), 4247–4253.
\bibitem{CaCo1} W. Castryck and F. Cools, \emph{Linear pencils encoded in the Newton polygon}, accepted for publication in Int. Math. Res. Not.
\bibitem{CaCo2} W. Castryck and F. Cools, \emph{The lattice size of a lattice polygon}, J. Comb. Theory Ser. A, Vol. 136 (2015), 64-95.
\bibitem{CCDL} W. Castryck, F. Cools, J. Demeyer and A. Lemmens, \emph{Computing graded Betti tables of toric surfaces}, preprint 2016 (arXiv:1606.08181).
\bibitem{DMN} J. Draisma, T. B. McAllister and B. Nill, \emph{Lattice-Width Directions and Minkowski's $3^d$-Theorem}, SIAM J. Discrete Math. 26 (2012), 1104–1107.
\bibitem{HZ} C. Haase and G.M. Ziegler, \emph{On the maximal width of empty lattice simplices}, Europ. J. Combinatorics 21 (2000), 111-119.
\bibitem{Seb} A. Seb\H{o}, \emph{An introduction to empty lattice-simplices}, Lecture Notes in Computer Science 1610, (Cornu\'ejols, Burkard, Woeginger Eds.), Springer, June 1999.
\bibitem{FTM} L. Fejes T\'oth and E. Makai Jr., \emph{On the thinnest non-separable lattice of convex plates}, Studia Scientiarum Mathematicarum Hungarica 9 (1974), 191-193.  
\end{thebibliography}
\end{document}